\documentclass[10pt]{article}
\usepackage{amsfonts}
\usepackage{amsmath}
\usepackage{amssymb}
\usepackage{multirow}
\usepackage{bbm}
\usepackage{empheq}
\usepackage{latexsym}
\usepackage{amscd}
\usepackage{bbm}
\usepackage{url}
\usepackage{xcolor}
\usepackage{hyperref}
\usepackage[left=3.2cm,right=3.2cm, top=2.5cm,bottom=2.5cm,bindingoffset=0cm]{geometry}
\usepackage{bbm}

\newenvironment{Proof}{\par\noindent{\it Proof.}}

\newcommand\addressT{\noindent\leavevmode
\medskip
\noindent
Kateryna Tatarko, \\
Department of Pure Mathematics,\\
University of Waterloo, \\
Waterloo, Ontario, Canada, N2L 3G1.\\
\texttt{\small
e-mail:  ktatarko@uwaterloo.ca}
}

\newcommand\addressW{\noindent\leavevmode
	\medskip
	\noindent
	Elisabeth M. Werner, \\
	Department of Mathematics,\\
	Case Western Reserve University,\\
	Cleveland, Ohio, USA, 44106. \\
	\texttt{\small
		e-mail:  elisabeth.werner@case.edu}
}

\date{}

\newcounter{theorem}[section]
\newtheorem{theorem}{Theorem}

\newtheorem{cor}[theorem]{Corollary}
\newtheorem{defi}[theorem]{Definition}
\newtheorem{prop}[theorem]{Proposition}

\newtheorem{Remark}[theorem]{Remark}

\newcommand{\R}{\mathbb{R}}

\newcommand{\qw}{\mathcal{W}}

\def\vol{{\rm vol}}

\def\1{\mathds{1}}

\begin{document}

\title{
Curvature functionals on convex bodies 
}

\author{Kateryna Tatarko\thanks{Partially supported by NSERC} \ and Elisabeth M. Werner
\thanks{Partially supported by  NSF grant  DMS-2103482}}

\date{}

\maketitle

\begin{abstract}
We investigate the {\em weighted $L_p$ affine surface areas} which appear  in the recently established $L_p$ Steiner formula of the $L_p$ Brunn Minkowski theory.
We show that they are valuations on the set of convex bodies and prove isoperimetric inequalities for them.
We show that they are related to $f$ divergences  of the cone measures  of the convex body and its polar, namely the
Kullback-Leibler divergence and  the R\'enyi-divergence.
	\end{abstract}
\vskip 3mm

\section {Introduction}
In \cite{TatarkoWerner2019} an  $L_p$ Steiner formula was proved for the $L_p$ affine surface area, namely, 
if a convex body $K$ is $C^2_+$, 
then we have for all suitable $t$ and for all $p \in \mathbb{R}$, $p \neq -n$, that 
\begin{eqnarray} \label{Formel}
as_p(K + t B^n_2) = \sum\limits_{k = 0}^\infty \left[ \sum\limits_{m=0}^k  \binom{ \frac{n(1-p)}{n+p}}{{k-m}} \qw^p_{{m}, k}(K) \right] t^k,
\end{eqnarray}
where  
\begin{equation} \label{ASA}
as_p(K)= \int\limits_{\partial K} \frac{H_{n-1}(x)^{\frac{p}{n+p}}}
{\langle x,N(x)\rangle ^{\frac{n(p-1)}{n+p}}} d\mathcal{H}^{n-1}
\end{equation}
is the $L_p$ affine surface area of a convex body $K$,  $N(x)$ is the outer normal to $K$ in $x \in \partial K$, the boundary of $K$, $H_{n-1}(x)$ is the Gauss curvature in $x$ and $\mathcal{H}^{n-1}$ is the usual surface area measure on $\partial K$. The Euclidean unit ball centered at $0$ is denoted by $B^n_2$.
\par
Identity (\ref{Formel}) is the analogue of the classical Steiner formula (e.g., \cite{Gardner, SchneiderBook}) of the Brunn Minkowski theory in the more recent
$L_p$ Brunn Minkowski theory. This theory  has as it's starting point E. Lutwak's  seminal paper \cite{Lutwak96} and it has developed immensely since
(e.g., \cite{BW2015, BW2016, BesauLudwigWerner, BoroLutwakYangZhang, GardnerHugWeilYe, HuangLutwakYangZhang, LudwigReitzner10, MeyerWerner, XingYe, Ye2012, Z3, Zhao}).
In analogy to the classical theory, the 
 coefficients $\qw^p_{{m}, k}(K) $ 
are called $L_p$ Steiner coefficients  and they are defined  in \cite{TatarkoWerner2019} 
for a (general) convex body $K$ in $\mathbb{R}^n$, for all $k, m \in \mathbb{N} \cup \{0\}$ as  
\begin{eqnarray*} 
	&&\hskip -11mm   \mathcal{W}^p_{m,\, k}(K) = \nonumber \\
	&&\hskip -11mm \int\limits_{\partial K} \langle x, N(x) \rangle^{m - k + \frac{n(1-p)}{n+p}} H_{n-1}^{\frac p{n+p}}  \sum_{\substack{
			i_1, \dots, i_{n-1} \geq 0 \\
			i_1 + 2i_2 + \dots + (n-  1)i_{n-1}=m}}
	c(n, p,m)  \prod\limits_{j = 1}^{n - 1}  {n - 1\choose j}^{i_j} H_{j}^{i_j} \, d\mathcal{H}^{n-1},
\end{eqnarray*}
where the $H_j$ are  the $j$-th normalized elementary symmetric functions of the principal curvatures.
The $c(n, p,m)$ are certain binomial coefficients, see \cite{TatarkoWerner2019, TatarkoWerner}  for the details.
The $L_p$ Steiner coefficients were studied  in \cite{TatarkoWerner}, where it was proved, among other results, that they are 
valuations on the set of convex bodies.

\subsection{Main Results}

In this paper, we look at the $L_p$ Steiner formula with a different focus. Expressions also appearing naturally in  formula (\ref{Formel}) are 
 {\em weighted $L_p$ affine surface areas}, $ \mu_{\vec{i}}-as_p(K)$, which we define in Section \ref{def}. 
 \vskip 2mm
 \noindent
 We investigate in detail those weighted $L_p$ affine surface areas in Section \ref{properties}.
We show that they are homogeneous of a certain degree and are invariant under rotations and reflections. 
We show that they  are valuations on the set of convex bodies.  Valuations have become a vitally import subject of study in convexity and affine and 
differential geometry, e.g., \cite{Alesker, ColesantiLudwigMussnig, Haberl, HaberlPara, Ludwig, Schu}.
The weighted $L_p$ affine surface areas satisfy isoperimetric inequalities which generalize the $L_p$ affine isoperimetric inequalities
of \cite{Lutwak96, WernerYe2008}. This is shown in Theorem \ref{p-aff-sio8}.
\vskip 2mm
\noindent
{\bf Theorem 3.2}
{\em Let $s\neq -n,  r \neq -n, t \neq -n$ be
real numbers. Let $K$ be a $C^2_+$ convex body in $\mathbb R^n$ with
centroid  at the origin. 
\vskip 2mm
(i) If $\frac{(n+r)(t-s)}{(n+t)(r-s)}>1$, then
\begin{equation}\label{i-2} \nonumber
	\mu_{\vec{i}} - as_r(K)  \leq \big(\mu_{\vec{i}} - as_t(K) 
	\big)^{\frac{(r-s)(n+t)}{(t-s)(n+r)}} \big(\mu_{\vec{i}} - as_s(K) \big)^{\frac{(t-r)(n+s)}{(t-s)(n+r)}}.
\end{equation}
\par
(ii) If $\frac{(n+r)t}{(n+t)r} >1$,  then
$$
\frac{\mu_{\vec{i}} - as_r(K)}{\mu_{\vec{i}}-\vol_n(K)}\leq n^{\frac{n(t-r)}{t(n+r)}}
\bigg(\frac{\mu_{\vec{i}} - as_t(K)}{\mu_{\vec{i}}-\vol_n(K)}\bigg)^{\frac{r(n+t)}{{t(n+r)}}}.
$$
Equality holds in the above inequalities, if and only if $K$ is an ellipsoid.}
\vskip 2mm
\noindent
We show in Section \ref{geo}  that the weighted $L_p$ affine surface areas have natural geometric interpretations in terms of certain convex bodies
associated with the given convex body $K$.
\vskip 3mm
\noindent
We prove a monotonicity behavior in the parameter $p$  for the weighted $L_p$ affine surface areas which allows to establish 
asymptotics for the weighted $L_p$ affine surface areas. These asymptotics  connects them to entropy powers, namely to the Kullbak-Leibler divergence $D_{KL}$ of the cone measures 
of $K$ and its polar $K^\circ$, $Q_K$ and $P_K$. We quote the relevant Theorem \ref{ASYM} and refer to Section \ref{KL} for the details.
We put 
\begin{equation*} \label{omegaK}
	\omega_{m,k,\vec{i}}^p (K) =
	\int\limits_{\partial K}\frac{H_{n-1}(x)^{\frac{p}{n+p}}}
	{\langle x,N(x)\rangle ^{\frac{n(p-1)}{n+p}}}  \,\langle x,N(x)\rangle ^{m-k}  \,  \prod\limits_{j = 1}^{n - 1}\left\{ {n - 1\choose j}^{i_j} H_{j}^{i_j}(x)\right\} \, d\mathcal{H}^{n-1}(x), 
\end{equation*}
and then  the following theorem holds.
\vskip 2mm
\noindent
{\bf Theorem 3.8} 	{\em Let $K$ be a $C^2_+$ convex body in $\mathbb R^n$ with
centroid  at the origin. Then
\vskip 2mm
(i) \begin{equation*}
		\lim\limits_{p \rightarrow \infty} \left(\frac{\omega_{m, k, \vec{i}}^p (K)}{\omega^\infty_{m, k, \vec{i}}(K)}\right)^{n+p} = \textup{exp} \left(-\frac{n \, D_{KL} (P_K||Q_K)}{\mu_{\vec{i}} - \vol_n(K^\circ)}\right).
	\end{equation*}
\vskip 2mm
(ii) \begin{equation*}
	\lim\limits_{p \rightarrow 0} \left(\frac{\omega_{m, k, \vec{i}}^p (K^\circ)}{\omega^0_{m, k, \vec{i}}(K^\circ)}\right)^{\frac{n(n+p)}p} = \textup{exp} \left(-\frac{n \, D_{KL} (P_{K^\circ}||Q_{K^\circ})}{\mu_{\vec{i}} - \vol_n(K^\circ)}\right).
\end{equation*}}
\vskip 3mm
\noindent
This leads naturally to consider more general $f$-divergences than just the  Kullbak-Leibler divergence.  We treat that in Section \ref{Section-fdiv}, where we also observe that the weighted $L_p$ affine surface areas themselves are special $f$-divergences.
\vskip 2mm
\noindent
Throughout the paper we  assume  that the convex bodies  $K$  are $C^2_+$, i.e.,  $K$ has twice continuously differentiable  boundary with strictly positive Gauss curvature everywhere and such that $0$ is the centroid of $K$, $0 = \frac{1}{\vol_n(K)} \int_{K} xdx$.

\section{Weighted $L_p$-affine surface areas}

\subsection {Background from differential geometry} \label {BackDG}

For more information and the details in this section we refer to e.g., \cite{Gardner,  SchneiderBook}.
\vskip 2mm
Let $K$ be a convex body of class $C^2.$ For a point $x$ on the boundary $\partial K$ of $K$  we denote by $N(x)$ the unique   outward unit normal vector of $K$ at $x.$ The map $N_K: \partial K \to S^{n - 1}$ is called the spherical image map or Gauss map of $K$ and is of class $C^1.$ Its differential is called the Weingarten map. The eigenvalues of the Weingarten map are the principal curvatures $k_i(x)$ of $K$ at $x.$

The $j$-th normalized elementary symmetric functions of the principal curvatures are denoted by $H_j$. They are defined as follows
\begin{equation}\label{ESFPC}
H_j = {n - 1 \choose j}^{-1} \sum_{1\leq i_1 < \dots < i_j \leq n - 1} k_{i_1} \cdots k_{i_j}
\end{equation}
for $j = 1, \dots, n-1$ and $H_0 = 1.$ Note that 
$$H_1 = \frac{1}{n-1}  \sum_{1\leq i \leq n - 1} k_{i} $$ is the mean curvature, that is the average of principal curvatures, and 
$$H_{n - 1}= \prod_{i=1}^{n-1} k_i$$ 
is the Gauss curvature. 
\par
We say that $K$ is of class $C^2_+$ if $K$ is of class $C^2$ and the Gauss map $\nu$ is a diffeomorphism. This means in particular that  $N_K$ has a smooth inverse. This assumption is stronger than just $C^2,$ and  is equivalent to the assumption that all principal curvatures are strictly positive, or that the Gauss curvature $H_{n-1} \ne 0.$
It also means that the differential of $N_K$, i.e., the Weingarten map, is of  maximal rank everywhere. 
\par
Let $K$ be of class $C^2_+$. For $u \in \R^n \setminus\{0\}$, let $\xi_K(u)$ be the unique point on the boundary of $K$ at which $u$ is an outward  normal vector. 
The map $\xi_K$ is defined on $\R^n \setminus \{0\}$. Its restriction to the sphere $S^{n - 1},$ the map $\bar{\xi}_K: S^{n - 1} \to \partial K$, is called the reverse spherical image map, or reverse Gauss map.  The differential of $\bar{\xi}_K$ is called the reverse Weingarten map. The eigenvalues of  the reverse Weingarten map are called the principal radii of curvature $r_1, \dots, r_{n - 1}$ of $K$ at $u\in S^{n - 1}.$ 
\par
The $j$-th normalized elementary symmetric functions of the principal radii of curvature are denoted by $s_j$. In particular, $s_0=1$,   and for $1 \leq j \leq n-1$ they are defined by
\begin{equation}\label{ESFPR}
s_j = {n - 1 \choose j}^{-1} \sum_{1\leq i_1 < \dots < i_j \leq n - 1} r_{i_1} \cdots r_{i_j}.
\end{equation}
Note that the principal curvatures are functions on the boundary of $K$ and the principal radii of curvature are functions on the sphere. 
\par
Now we describe  the connection between $H_j$ and $s_j.$ For a body $K$ of class $C^2_+$, we have for $u\in S^{n - 1}$ that $\bar{\xi}_K(u) = N_K^{-1}(u)$. 
In particular, the principal radii of curvature are reciprocals of the principal curvatures, that is
$$
r_i(u) = \frac1{k_i(\bar{\xi}_K(u))}.
$$
This implies that for $x \in \partial K$ with $\nu(x)=u$, 
$$
s_j = {n - 1 \choose j}^{-1} \sum_{1\leq i_1 < \dots < i_j \leq n - 1} \frac1{k_{i_1}(\bar{\xi}_K(u)) \cdots k_{i_j}(\bar{\xi}_K(u))} = \frac{H_{n - 1 -j}}{H_{n - 1}}\Big(\bar{\xi}_K(u)\Big)
$$
and
$$
H_j  =  \frac{s_{n - 1 -j}}{s_{n - 1}}\Big(\nu(x)\Big), 
$$
for $j = 1, \dots, n-1.$
\par

\subsection{Definitions} \label{def}

For fixed $k \in \mathbb{N}$, $m\in \mathbb{N} \cup \{0\}$ and fixed sequence $\vec{i}=\{i_j\}_{j=0}^{n-1}$ such that $i_1+2 i_2+ \cdots (n-1) i_{n-1} =m$ and all $p \in \mathbb{R}$, $p \neq -n$,
we define 
\begin{equation} \label{omegaK}
	\omega_{m,k,\vec{i}}^p (K) =
	\int\limits_{\partial K}\frac{H_{n-1}(x)^{\frac{p}{n+p}}}
	{\langle x,N(x)\rangle ^{\frac{n(p-1)}{n+p}}}  \,\langle x,N(x)\rangle ^{m-k}  \,  \prod\limits_{j = 1}^{n - 1}\left\{ {n - 1\choose j}^{i_j} H_{j}^{i_j}(x)\right\} \, d\mathcal{H}^{n-1}(x)
\end{equation}
and 
\begin{equation} \label{omegaS}
	u_{m,k,\vec{i}}^p (K) =
	\int\limits_{S^{n-1}}\frac{f_{K}(u)^{\frac{n}{n+p}}}
	{h_K(u) ^{\frac{n(p-1)}{n+p}}}  \, h_K(u)  ^{m-k}  \,  \frac{ \prod\limits_{j = 1}^{n - 1} {n - 1\choose j}^{i_j}  s_{n-1-j}^{i_j} (u)}{s_{n-1}^{\sum\limits_j i_j+\frac p{n+p}}(u)} \, d\mathcal{H}^{n-1}(u).
\end{equation}

\begin{Remark}
	Note that the above quantities are vanishing for polytopes. Therefore, we will treat only $C^2_+$ convex bodies throughout the text and then	
	\begin{equation}\label{gleich}
	\omega_{m,k,\vec{i}}^p (K)= u_{m,k,\vec{i}}^p (K).
	\end{equation}
\end{Remark}
\vskip 3mm
\noindent
Denote
$$
c(n) = \prod\limits_{j = 1}^{n - 1} {n - 1\choose j}^{i_j}.
$$
Let $\mu_{m,k,\vec{i}}$ be the measure on $\partial K$ with density
$$
d\mu_{m,k,\vec{i}} (x) = c(n) \, \langle x,N(x)\rangle^{m-k} \prod\limits_{j = 1}^{n - 1} H_{j}^{i_j}(x)\, d\mathcal{H}^{n-1}(x)
$$ 
with respect to the surface measure $\mathcal{H}^{n-1}$ on $\partial K$ and let $\nu_{m,k,\vec{i}}$ be the measure on $S^{n-1}$ with density
$$
d\nu_{m,k,\vec{i}} (u) =c(n) \,  h^{m-k}_K(u)   \frac{\prod\limits_{j = 1}^{n - 1}s_{n-1-j}^{i_j} (u)}{s_{n-1}^{\sum\limits_j i_j+\frac p{n+p}}(u)} d\mathcal{H}^{n-1}(u).
$$
with respect to the surface measure $\mathcal{H}^{n-1}$ on $S^{n-1}$.
To keep notations simple we mostly write $\mu_{\vec{i}}$ and $\nu_{\vec{i}}$ instead of $\mu_{m,k,\vec{i}}$ and $\nu_{m,k,\vec{i}}$.
\vskip 3mm
\noindent
We then define the  {\em weighted $L_p$-affine surface areas} by
\begin{equation} \label{WASA}
 \mu_{\vec{i}} - as_p(K)= \omega_{m,k,\vec{i}}^p (K) = \int\limits_{\partial K} \frac{H_{n-1}(x)^{\frac{p}{n+p}}}
{\langle x,N(x)\rangle ^{\frac{n(p-1)}{n+p}}} d\mu_{\vec{i}}(x).
\end{equation}
Definition (\ref{ASA}) explains that those can be considered as $L_p$-affine surface area weighted by the measure $\mu_{\vec{i}}$. 
In particular, 
$$
\omega_{m,k,\vec{i}}^0 (K) = \int\limits_{\partial K} \langle x,N(x)\rangle d\mu_{\vec{i}}(x) = n(\mu_{\vec{i}}-\vol_n(K))
$$
is a weighted volume of $K$, weighted by the measure $\mu_{\vec{i}}$
and
$$
\omega_{m,k,\vec{i}}^\infty (K) = \int\limits_{\partial K} \frac{H_{n-1}}{\langle x,N(x)\rangle^n} d\mu_{\vec{i}}(x) =n( \mu_{\vec{i}} - \vol_n(K^\circ))
$$
is a weighted volume of $K^\circ$, weighted by the measure $\mu_{\vec{i}}$, where 
$$
K^\circ=\{ y \in \mathbb{R}^n: \langle y, x \rangle \leq 1, \text{ for all } x \in K \}
$$
is the polar body of $K$.
\vskip 3mm
\noindent
We can take another point of view for the measure $d\mathcal{H}^{n-1}$ on $S^{n-1}$ via the density $f_p(K,u) = \left(\frac{f_K(u)}{h_K^{p-1}(u)}\right)^{\frac n{n+p}}$. This density was introduced by Lutwak \cite{Lutwak96}. Then
$$
u_{m,k,\vec{i}}^p (K) = c(n)
\int\limits_{S^{n-1}} f_p(K,u) h_K(u)  ^{m-k}  \,  \frac{ \prod\limits_{j = 1}^{n - 1}  s_{n-1-j}^{i_j} (u)}{s_{n-1}^{\sum\limits_j i_j+\frac p{n+p}}(u)} \, d\mathcal{H}^{n-1}(u).
$$

\subsection{Special cases}

\noindent 
Note that 
$$
\omega_{0, 0, 0}^p(K) = u_{0,0,0}^p(K) = as_p(K),
$$
where we denote the sequence $\vec{i} = \{0, \dots, 0\}$ as $0$.
\vskip 2mm
1. When $k = m$, we get
$$
\omega_{m,m,\vec{i}}^p (K) =
\int\limits_{\partial K}\frac{H_{n-1}(x)^{\frac{p}{n+p}}}
{\langle x,N(x)\rangle ^{\frac{n(p-1)}{n+p}}}  \,  \prod\limits_{j = 1}^{n - 1}\left\{ {n - 1\choose j}^{i_j} H_{j}^{i_j}(x)\right\} \, d\mathcal{H}^{n-1}(x).
$$
\par
2. $m = 0$ implies that $\vec{i} = 0$ and  \eqref{omegaK} simplifies to 
\par
$$
\omega_{0,k,0}^p (K) =
\int\limits_{\partial K}\frac{H_{n-1}(x)^{\frac{p}{n+p}}}
{\langle x,N(x)\rangle ^{\frac{n(p-1)}{n+p}+k}}   \, d\mathcal{H}^{n-1}(x) = as_{p + \frac kn (n+p),\, -k}(K) 
$$
(see \cite[Eq. (29)]{TatarkoWerner2019}).
\par
3. For the Euclidean unit ball $B^n_2$, we get
\begin{align}\label{ball}
	\omega_{m,k,\vec{i}}^p (B^n_2) = \vol_{n-1}(B^n_2) \prod\limits_{j = 1}^{n - 1} {n - 1\choose j}^{i_j} = c(n)\, \vol_{n-1}(\partial B^n_2)
\end{align}
which does not depend on $k$. Note that if $i_0 \ne 0$ and $i_1 = \dots = i_{n-1} = 0$, that is $\vec{i} = \{i_0, 0, \dots, 0\}$, then $\omega_{m,k,\vec{i}}^p (B^n_2) = \vol_{n-1}(\partial B^n_2)$.
\par
4. If $p = 0$, we have
$$
\omega_{m,k, \vec{i}}^0(K) = c(n)
\int\limits_{\partial K} 
\langle x,N(x)\rangle ^{m-k+1}  \,  \prod\limits_{j = 1}^{n - 1}H_{j}^{i_j}(x) \, d\mathcal{H}^{n-1}(x).
$$
\par
When in addition $k = m$, we get 
$$
\omega_{m,m, \vec{i}}^0(K) = c(n)
\int\limits_{\partial K} 
\langle x,N(x)\rangle  \,  \prod\limits_{j = 1}^{n - 1} H_{j}^{i_j}(x) \, d\mathcal{H}^{n-1}(x).
$$
\par
5. If $p = \infty$, we get that
$$
\omega_{m,m, \vec{i}}^\infty(K) = c(n)
\int\limits_{\partial K} \frac{H_{n-1}}{\langle x,N(x)\rangle^n} 
 \,  \prod\limits_{j = 1}^{n - 1} H_{j}^{i_j}(x) \, d\mathcal{H}^{n-1}(x).
$$

\bigskip

\section{Properties of the weighted $L_p$-affine surface areas} \label{properties}
\vskip 3mm
\subsection{Valuation, Invariance and Homogeneity}
\vskip 3mm

\begin{prop}
Let $K$ be a $C^2_+$ convex body in $\mathbb R^n$ with
centroid  at the origin.  Let $p \in \mathbb{R}$, $p\neq -n$.
Then $\omega_{m,k,\vec{i}}^p (K)$ is a $n\frac{n-p}{n+p}-k$-homogeneous valuation that is  invariant under rotations and reflections.
\end{prop}
\vskip 2mm
\noindent
\begin{proof}
The proof follows immediately from results in \cite{TatarkoWerner}. We present an outline of the proof for completeness. 
\vskip 2mm
\noindent
1. {\it Valuation.} As was shown in \cite[Theorem 5.9]{TatarkoWerner}, for all $1 \leq i_1,  \dots, i_{n-1} \leq n-1$, $1\leq j \leq n-1$, and $\alpha_1, \dots, \alpha_{j} \geq 0$,   
	$$
	\int\limits_{\partial K}\frac{H_{n-1}(x)^{\frac{p}{n+p}}}
	{\langle x,N(x)\rangle ^{\frac{n(p-1)}{n+p}}}  \,\langle x,N(x)\rangle ^{k-m}  \,  \prod\limits_{i = 1}^{j}  k_{i_j}^{\alpha_j}\, d\mathcal{H}^{n-1}(x)
	$$
	is a valuation. It immediately follows that $	\omega_{m,k,\vec{i}}^p (K) $ are valuations as the linear combination of valuations is again a valuation. 

\medskip

\noindent 2. {\it Homogeneity.} Similarly to \cite[Theorem 5.1]{TatarkoWerner} we can show that $\omega_{m, k, \vec{i}}^p (K)$ are homogeneous of order $n\frac{n-p}{n+p} - k$. Applying \cite[Proposition 5.4]{TatarkoWerner} with $T = a\, Id$, we get that $\omega_{m, k, \vec{i}}^p (K) = a^{k - n\frac{n-p}{n+p}} \omega_{m, k, \vec{i}}^p (aK)$.

\medskip

\noindent 3. {\it Invariance.} If $T$ is a rotation or a reflection, then $|\det T|=1$, $\|T^{-1t}(N_{ K}(T^{-1}(y)))\| = \| N_{ K}(T^{-1}(y))\| = 1$ and 
for all $1 \leq j \leq n-1$, 
$$\{ H_{j}(y): y \in \partial T(K) \}= \{H_{j}\left(x \right): \,  x \in \partial K \}.$$
Thus, using these observations and \cite[Proposition 5.4]{TatarkoWerner}, we get	
\begin{eqnarray*}
	&&\omega_{m, k, \vec{i}}^p (K) =\\
	&& \int\limits_{\partial T(K)} \langle y, N_{T(K)}(y) \rangle^{m - k + \frac{n(1-p)}{n+p}} H_{n-1}^{\frac p{n+p}}  (y) \prod\limits_{j = 1}^{n - 1}  {n - 1\choose j}^{i_j} H_{j}^{i_j}\left(y \right) \, d\mathcal{H}^{n-1}(y) = \omega_{m, k, \vec{i}}^p (T(K)).
\end{eqnarray*}
\end{proof}

\subsection{Inequalities }

\begin{theorem}\label{p-aff-sio8}
Let $s\neq -n,  r \neq -n, t \neq -n$ be
real numbers. Let $K$ be a $C^2_+$ convex body in $\mathbb R^n$ with
centroid  at the origin. 
\vskip 2mm
(i) If $\frac{(n+r)(t-s)}{(n+t)(r-s)}>1$, then
\begin{equation}\label{i-2} \nonumber
	\mu_{\vec{i}} - as_r(K)  \leq \big(\mu_{\vec{i}} - as_t(K) 
	\big)^{\frac{(r-s)(n+t)}{(t-s)(n+r)}} \big(\mu_{\vec{i}} - as_s(K) \big)^{\frac{(t-r)(n+s)}{(t-s)(n+r)}}.
\end{equation}
\par
(ii) If $\frac{(n+r)t}{(n+t)r} >1$,  then
$$
\frac{\mu_{\vec{i}} - as_r(K)}{\mu_{\vec{i}}-\vol_n(K)}\leq n^{\frac{n(t-r)}{t(n+r)}}
\bigg(\frac{\mu_{\vec{i}} - as_t(K)}{\mu_{\vec{i}}-\vol_n(K)}\bigg)^{\frac{r(n+t)}{{t(n+r)}}}.
$$
Equality holds in the above inequalities, if and only if $K$ is an ellipsoid.
\end{theorem}

\vskip 2mm 
\noindent {\bf Proof}
\par
\noindent
(i)   By
H\"older's inequality - which enforces the condition $\frac{(n+r)(s-t)}{(n+t)(s-r)} >1$, we then get
 \begin{eqnarray} \label{firstcase}
\mu_{\vec{i}} - as_r(K) = \omega_{m,k,\vec{i}}^r (K) &=&\int\limits_{\partial K}\frac{H_{n-1}(x)^{\frac{r}{n+r}}}
{\langle x,N(x)\rangle ^{\frac{n(r-1)}{n+r}}}  \,\langle x,N(x)\rangle ^{m-k}  \,  \prod\limits_{j = 1}^{n - 1}\left\{ {n - 1\choose j}^{i_j} H_{j}^{i_j}(x)\right\} \, d\mathcal{H}^{n-1}(x) \nonumber
\\
&=& \int _{\partial K}
\left(\frac{H_{n-1}(x)^{\frac{t}{n+t}}}{\langle x, N(x)\rangle
^{\frac{n(t-1)}{n+t}}} \right)^{\frac{(r-s)(n+t)}{(t-s)(n+r)}}
\left(\frac{H_{n-1}(x)^{\frac{s}{n+s}}}{\langle x, N(x)\rangle
^{\frac{n(s-1)}{n+s}}}\right)^{\frac{(t-r)(n+s)}{(t-s)(n+r)}}\,
d\mu_{\vec{i}}(x) \nonumber
\\ &\leq & \big(\omega_{m,k,\vec{i}}^t (K)
\big)^{\frac{(r-s)(n+t)}{(t-s)(n+r)}} \big(\omega_{m,k,\vec{i}}^s (K)\big)^{\frac{(t-r)(n+s)}{(t-s)(n+r)}} \nonumber \\
&=& \big(\mu_{\vec{i}} - as_t(K) 
\big)^{\frac{(r-s)(n+t)}{(t-s)(n+r)}} \big(\mu_{\vec{i}} - as_s(K) \big)^{\frac{(t-r)(n+s)}{(t-s)(n+r)}}. \nonumber
 \end{eqnarray}
\par
\noindent (ii) Similarly, again using H\"older's inequality -which now
enforces the condition $~\frac{(n+r)t}{(n+t)r} >1,~$
\begin{eqnarray*}
\mu_{\vec{i}} - as_r(K) =\omega_{m,k,\vec{i}}^r (K)&=&\int\limits_{\partial
K} \frac{H_{n-1}(x)^{\frac{r}{n+r}}}{\langle x, N(x)\rangle
^{\frac{n(r-1)}{n+r}}}\, d\mu_{\vec{i}} (x)  = \int\limits_{\partial K}
\left(\frac{H_{n-1}(x)^{\frac{t}{n+t}}}{\langle x, N(x)\rangle
^{\frac{n(t-1)}{n+t}}} \right)^{\frac{r(n+t)}{t(n+r)}} \frac{d\mu_{\vec{i}}
(x)}{\langle x, N(x)\rangle ^{\frac{(r-t)n}{(n+r)t}}} \\ 
&\leq &\big(\omega_{m,k,\vec{i}}^t(K)\big)^{\frac{r(n+t)}{t(n+r)}} \big(
\omega_{m,k,\vec{i}}^0(K)\big)^{\frac{(t-r)n}{(n+r)t}} \nonumber \\
&=&
\bigg(\mu_{\vec{i}} - as_t(K)\bigg)^{\frac{r(n+t)}{{t(n+r)}}} \bigg(n(\mu_{\vec{i}}-\vol_n(K))\bigg)^{\frac{n(t-r)}{{t(n+r)}}}.
\end{eqnarray*}
\vskip 2mm
\noindent
The equality characterizations follow from the equality characterization of H\"{o}lder's inequality. 
\par
\noindent
Equality holds in {\it(i)} and {\it (ii)} if and only if equality holds in H\"{o}lder's inequality which happens if and only if 
$$
\frac{H_{n-1}}{\langle x,N(x)\rangle^{n+1}} = \textup{constant}
$$
$\mu_{\vec{i}}$-almost everywhere on $\partial K$. As $\partial K$ is $C^2_+$, $\{ x \in \partial K: \mu_{\vec{i}}(x) =0\} = \emptyset$
and therefore 
$$
\frac{H_{n-1}}{\langle x,N(x)\rangle^{n+1}} = \textup{constant}
$$ 
holds for all $x \in \partial K$.
Thus we can use the following theorem by Petty \cite{Petty}, which then finishes the proof.
\vskip 2mm
\noindent
\begin{theorem}[\cite{Petty}] \label{Petty}
	Let $K$ be a convex body in $\R^n$ that is $C^2_+$. $K$ is an ellipsoid if and only if for all $x$ in $\partial K$
	$$
	\frac{H_{n-1}}{\langle x,N(x)\rangle^{n+1}} = c
	$$
	where $c>0$ is a constant. 
\end{theorem}

\begin{theorem}\label{p-aff-sio9}
	Let $K$ be a $C^2_+$ convex body in $\mathbb R^n$ with
	centroid at the origin. 
	\vskip 2mm
	If $r<s<k$ then
	\begin{equation} \nonumber
		\omega_{m,s,\vec{i}}^p (K) \leq (\omega_{m,k,\vec{i}}^p (K))^{\frac{k-s}{k-r}} (\omega_{m,r,\vec{i}}^p (K))^{\frac{s-r}{k-r}}
	\end{equation}

	Equality holds if and only if $K$ is a ball.
\end{theorem}

\vskip 2mm 
\noindent 
{\bf Proof}
\begin{eqnarray}
	 &&\omega_{m,s,\vec{i}}^p (K) =\int\limits_{\partial K}\frac{H_{n-1}(x)^{\frac{p}{n+p}}}
	{\langle x,N(x)\rangle ^{\frac{n(p-1)}{n+p}}}  \,\langle x,N(x)\rangle ^{m-s}  \,  \prod\limits_{j = 1}^{n - 1}\left\{ {n - 1\choose j}^{i_j} H_{j}^{i_j}(x)\right\} \, d\mathcal{H}^{n-1}(x) \nonumber
	\\
	&=& \int\limits_{\partial K}
	\left(\frac{H_{n-1}(x)^{\frac{p}{n+p}}}{\langle x, N(x)\rangle
		^{\frac{n(p-1)}{n+p}}} \langle x,N(x)\rangle ^{m-k} c(n) \prod\limits_{j = 1}^{n - 1} H_{j}^{i_j}\right)^{\frac{s-r}{k-r}}
	\left(\frac{H_{n-1}(x)^{\frac{p}{n+p}}}{\langle x, N(x)\rangle
		^{\frac{n(p-1)}{n+p}}} \langle x,N(x)\rangle ^{m-r} c(n) \prod\limits_{j = 1}^{n - 1} H_{j}^{i_j}\right)^{\frac{k-s}{k-r}}
	d\mathcal{H}^{n-1} \nonumber\\ 
	&\leq & \big(\omega_{m,k,\vec{i}}^p (K) 
	\big)^{\frac{s-r}{k-r}} \big(\omega_{m,r,\vec{i}}^p (K) \big)^{\frac{k-s}{k-r}}. \nonumber
\end{eqnarray}

\vskip 3mm
\noindent
\begin{Remark} 
\par
 When $\vec{i} = 0$ and $k = 0$, we recover $L_p$ affine isoperimetric inequalities of \cite{WernerYe2008}.
\end{Remark}
\vskip 3mm
\noindent
We also obtain monotonicity behaviors. 

\begin{cor}\label{monotonicity}
	Let $K$ be a $C^2_+$ convex body in $\R^n$ with centroid at the origin.  Let $p \ne -n$ be a real number.
	
	(i) The function $p \rightarrow \left(\frac{\omega_{m, k, \vec{i}}^p (K)}{\omega^0_{m, k, \vec{i}}(K)}\right)^{\frac{n+p}{p}}$ is increasing in $p \in (-n, \infty)$ and $p \in (-\infty, -n)$. 
	
	(i) The function $p \rightarrow \left(\frac{\omega_{m, k, \vec{i}}^p (K)}{\omega^0_{m, k, \vec{i}}(K)}\right)^{n+p}$ is decreasing in $p \in (-n, \infty)$ and $p \in (-\infty, -n)$.
	
	(iii) The inequalities are strict  unless $K$ is an ellipsoid.
\end{cor}
\begin{Proof}
	The statement {\it (i)} follows immediately from Theorem~\ref{p-aff-sio8} {\it (ii)}. The statement {\it (ii)} follows from Theorem~\ref{p-aff-sio8} {\it (i)}, by letting $s \rightarrow \infty$. The statement {\it (iii)} follows from the equality characterizations.  
\end{Proof}

\vskip 4mm
 
\subsection{Geometric interpretations} \label{geo}

We recall several constructions of convex bodies associated with a given convex body $K$. Namely, 
\vskip 3mm
1. {\em Weighted floating bodies} \cite{Werner2002}
\vskip 2mm
\noindent
Let $K$ be a convex body in $ {\mathbb R}^n$ and denote by $\lambda$ the Lebesgue measue on $\mathbb{R}^n$.
Let $s \geq 0$ and let $f:K\rightarrow\mathbb R$ be an
integrable function such that $f >0$ $\lambda$-a.e. 
The weighted floating body $F(K,f,s)$ was defined in \cite{Werner2002} (see also \cite{BW2015, BW2016, BesauLudwigWerner}) as the intersection of all 
closed half-spaces
$H^{+}$
whose defining hyperplanes $H$ cut off a set of $(f\, \lambda)$-measure less than
or equal to $s$ from $K$,  
\begin{equation*}\label{weighted floating body}
F(K,f,s)=\bigcap_{\int_{ K\cap H^{-}} f\, d\lambda \leq s}H^{+}.
\end{equation*}
It was shown in \cite{Werner2002} that
$$
2 \bigg(\frac{\mbox{vol}_{n-1}(B^{n-1}_2)}{n+1}\bigg)^\frac{2}{n+1} \lim_{s \to 0} \frac{\mbox{vol}_n(K)-\mbox{vol}_n(F(K,f,s))}
{s^\frac{2}{n+1}}=
\int_{\partial K} \frac{H_{n-1}^\frac{1}{n+1}} {f^\frac{2}{n+1}} \,  d\mathcal{H}^{n-1}. 
$$
\vskip 2mm
2. {\em Surface bodies} \cite{SchuettWerner2004}
\vskip 2mm
\noindent
Let $K$ be a convex body and $f:\partial K\rightarrow\mathbb R$ be a nonnegative,
integrable function with $\int_{\partial K}f \, \mathcal{H}^{n-1}=1$. The probability
measure $\mathbb P_{f}$ is the measure on $\partial K$ with density $f$.
Let $s \geq 0$.  The surface body $S(K, f,s)$ was defined in \cite{SchuettWerner2004} as the intersection of all the closed half-spaces
$H^{+}$
whose defining hyperplanes $H$ cut off a set of $\mathbb P_{f}$-measure less than
or equal to $s$ from
$\partial K$, 
\begin{equation*}\label{surfacebody}
S(K,f,s)=\bigcap_{\mathbb P_{f}(\partial K\cap H^{-})\leq s}H^{+}.
\end{equation*}
It was shown in \cite{SchuettWerner2004} that
$$
2 \bigg(\mbox{vol}_{n-1}(B^{n-1}_2)\bigg)^\frac{2}{n-1}\lim_{s \to 0}
\frac{\mbox{\rm vol}_n(K)-\mbox{\rm vol}_n(S(K, f,s))}
{s^\frac{2}{n-1}}=
\int_{\partial K} \frac{H_{n-1}^\frac{1}{n-1}}
{f^\frac{2}{n-1}} \,  d\mathcal{H}^{n-1}. 
$$
\vskip 2mm
3. {\em Random polytopes} \cite{SchuettWerner2003}
\vskip 2mm
\noindent
A random polytope is the
convex hull of finitely many points that are chosen  with respect to
a probability measure.
In \cite{SchuettWerner2003}, random polytopes are considered where the points are
chosen from $\partial K$
with respect to $\mathbb P_{f}$, where  $f:\partial K\rightarrow\mathbb R$ is an integrable, nonnegative function
with $\int_{\partial K} f \, d\mathcal{H}^{n-1} =1$ and $ d \,\mathbb{P}_{f} = f \,d\mathcal{H}^{n-1}$.
Then the expected volume of such a random polytope is 
$$\mathbb E(f,N)=\mathbb E(\mathbb P_{f},N)=\int_{\partial
K}\cdots\int_{\partial K}\mbox{vol}_{n}([x_{1},\dots,x_{N}])
d\mathbb P_f(x_{1})\dots d\mathbb P_f(x_{N}),
$$
where $[x_{1},\dots,x_{N}]$ is the convex hull of the points $x_{1},\dots,x_{N}$.
It was shown in \cite{SchuettWerner2003} that under mild smoothness assumptions on $K$, 
$$ 
\lim_{N \to \infty} \frac{\mbox{\rm vol}_n(K)-\mathbb{E}(f
,N)}{\left(\frac{1}{N}\right)^\frac{2}{n-1}}=
c_n \int_{\partial K} \frac{H_{n-1}^\frac{1}{n-1}}
{f^\frac{2}{n-1}} \,  d\mathcal{H}^{n-1}. 
$$
$c_n=\frac{(n-1)^{\frac{n+1}{n-1}}\Gamma \left(n+1+\tfrac{2}{n-1}\right)}
{2(n+1)!(\mbox{\rm vol}_{n-2}(\partial B_{2}^{n-1}))^{\frac{2}{n-1}}}$

\vskip 2mm
4. {\em Ulam floating bodies} \cite{HuangSlomkaWerner}
\vskip 2mm
\noindent
Given a Borel measure $\mu$ on $\R^{n}$, the {\em
metronoid} associated to $\mu$ was introduced by Huang and Slomka \cite{HuangSlomka} and  is the convex set defined by 
\[
\mathcal{M}(\mu)=\bigcup_{\substack{0\leq f\leq1,\\
\int_{\R^{n}}f d{\mu}=1}}
\left\{ \int_{\R^{n}}yf\left(y\right) d{\mu}\left(y\right)\right\} ,
\]
where the union is taken over all functions $0\le f\le1$ for which
$\int_{\R^{n}}f d{\mu=1}$ and $\int_{\R^{n}}yf\left(y\right)d{\mu}\left(y\right)$
exists. The metronoid $\mathcal{M}_{\delta}(K)$ is generated by
the uniform measure on $K$ with total mass $\delta^{-1}\left|K\right|$.
Namely, let $\mathbbm{1}_{K}$ be the characteristic function of $K$, and
$\mu$ the measure whose density with respect to Lebesgue measure
is $\delta^{-1}\mathbbm{1}_{K}$. Then $\mathcal{M}_{\delta}(K):=\mathcal{M}(\mu)$.
In \cite{HuangSlomkaWerner}  weighted variations of $\mathcal{M}_{\delta}(K)$ were defined where the
weight is given by a positive continuous function $f:K\to\R$,  
\[
\mathcal{M}_{\delta}(K, f):=\mathcal{M}\left(\frac{f\left(x\right)}{\delta} \mathbbm{1}_{K}\left(x\right) d x \right).
\]
It was shown in \cite{HuangSlomkaWerner} that
\begin{equation*}
\lim_{\delta\searrow0}\frac{ \mbox{vol}_{n}\left (K\right)- \mbox{vol}_{n}\left(\mathcal{M}_{\delta}(K, f)\right)}{\delta^{\frac{2}{n+1}}}=
2\frac{n+1}{n+3}\left(\frac{\vol_{n-1}(B_{2}^{n-1})}{n+1}\right)^{\frac{2}{n+1}}
\int_{\partial K} 
\frac{H_{n-1}^\frac{1}{n+1}} {f^\frac{2}{n+1}} \,  d\mathcal{H}^{n-1}. 
\end{equation*}
\vskip 3mm
\noindent
This leads to  geometric interpretations of the weighted $L_p$-affine surface areas in terms of these associated bodies.
That is, if we let 
$$
f(x)= \frac{H_{n-1}(x)^{\frac{n(1-p)}{2(n+p)}}}
	{\left[c(n)\, H_{j}^{i_j}(x)\right]^\frac{n+1}{2}}  \,\langle x,N(x)\rangle ^{\left(\frac{n(p-1)}{n+p} +k-m\right) \frac{n+1}{2} }  
	$$	
in the case of weighted floating bodies and Ulam floating bodies, and 
$$
f(x)= \frac{H_{n-1}(x)^{\frac{2p+n(1-p)}{2(n+p)}}}
	{\left[c(n)\, H_{j}^{i_j}(x)\right]^\frac{n-1}{2}}  \,\langle x,N(x)\rangle ^{\left(\frac{n(p-1)}{n+p} +k-m\right) \frac{n-1}{2} }  
	$$	
in the case of surface bodies and random polytopes, respectively, and denote by $K_{f,s}$ the corresponding associated body, i.e., 
$$
 K_{f,s}= F(K, f,s), \hskip 2mm \text{or} \hskip 2mm K_{f,s}= S(K, f,s), \hskip 2mm \text{or} \hskip 2mmK_{f,s}=  \mathcal{M}_{s}(K, f), \hskip 2mm \text{or} \hskip 2mm
\vol_n(K_{f,s}) = \mathbb E(f,N), 
$$
then, with the properly adjusted constant $c_n$, we get the following proposition.
\vskip 2mm
\begin{prop}
Let $K$ be a $C^2_+$ convex body in $\R^n$ with centroid at the origin.  Let $p \ne -n$ be a real number.
\begin{equation*}
c_n  \lim_{s \to 0}
\frac{\mbox{\rm vol}_n(K)-\mbox{\rm vol}_n(K_{f,s})}
{s^\frac{2}{n-1}}=
\omega_{m,k,\vec{i}}^p (K).
\end{equation*}
\end{prop}
\vskip 3mm
\subsection{Asymptotics for the weighted $L_p$-affine surface areas}\label{KL}

Let $(X, \mu)$ be a measure space and $dP = p d\mu$ and $Q = q d\mu$ be measures in $X$ that are absolutely continuous with respect to the measure $\mu$. The {\it Kullback-Leibler divergence} or {\it relative entropy} from $P$ to $Q$ is defined as in \cite{CoverThomas}
\begin{equation}
	D_{KL}(P||Q) = \int\limits_X p \log{\frac pq} d\mu.
\end{equation}
The {\it information inequality} (also called {\it Gibb's inequality}) \cite{CoverThomas} holds for the Kullback-Leibler divergence. Let $P$ and $Q$ be as above, then
\begin{equation}
D_{KL}(P||Q) \geq 0.
\end{equation}
with equality if and only if $P=Q$.
\vskip 2mm
\noindent
We will apply this when $(X, \mu) = (\partial K, \mu_{\vec{i}})$, where $K$ is $C^2_+$  and densities $p$ and $q$ with respect to $\mu_{\vec{i}}$ given by
\begin{equation} \label{densities}
	p_K(x) = \frac{H_{n-1}}{\langle x,N(x)\rangle^n}, \qquad q_K(x) = \langle x,N(x)\rangle.
\end{equation} 
We let
\begin{equation}
	P_K = \frac{H_{n-1}}{\langle x,N(x)\rangle^n} \mu_{\vec{i}}, \qquad Q_K = \langle x,N(x)\rangle \mu_{\vec{i}}.
\end{equation} 
Recall that classical cone measure $cm_K$ on $\partial K$ is defined as follows: for every measurable set $A \subseteq \partial K$ 
$$
cm_K(A) = \vol_n(\{t a: \ a \in A,\, t\in [0, 1]\}).
$$
It is well-known (see, e.g, \cite{PaourisWerner2011} for a proof) that the measures $\frac{H_{n-1}}{\langle x,N(x)\rangle^n}\mathcal{H}^{n-1}$ and $\langle x,N(x)\rangle\mathcal{H}^{n-1}$ are the cone measures of $K$ and $K^\circ$,
\begin{equation*}
	cm_K(A) = \frac1n \int\limits_{A}  \langle x,N(x)\rangle \mathcal{H}^{n-1}(x), \qquad cm_{K^\circ}(A) = \frac1n \int\limits_{A}  \frac{H_{n-1}}{\langle x,N(x)\rangle^n}  \mathcal{H}^{n-1}(x).
\end{equation*}
The interpretation of $cm_{K^\circ}$ as the ``cone measure of $K^\circ$" is via the Gauss map on $K^\circ$, $N_{K^\circ}: \partial K^\circ \rightarrow S^{n-1}, \ x \longmapsto N_{K^\circ}(x)$ and the inverse of the Gauss map on $K$  $N_{K}: \partial K \rightarrow S^{n-1}, \ x \longmapsto N_{K}(x)$,
$$
N_K^{-1}N_{K^\circ} cm_{K^\circ}
$$
where we use $N_K(x)$ to emphasize that it is the normal vector of a body
$K$ at $x \in \partial K$.

\noindent We now define the {\it weighted cone measures} $\mu_{\vec{i}} -cm_K$ of $K$ and $\mu_{\vec{i}} -cm_{K^\circ}$ of $K^\circ$ by
\begin{equation}
	\mu_{\vec{i}} -cm_K = \frac1n \int\limits_{A} \langle x,N(x)\rangle \mu_{\vec{i}}(x),
\end{equation}
\begin{equation}
	\mu_{\vec{i}} -cm_{K^\circ} (A )= \frac1n \int\limits_{A}  \frac{H_{n-1}}{\langle x,N(x)\rangle^n} \mu_{\vec{i}}(x).
\end{equation}

\noindent We show that the limits of the weighted $L_p$ affine surface areas  are entropy powers. 

\begin{theorem} \label{ASYM}
Let $K$ be a $C^2_+$ convex body in $\mathbb R^n$ with
centroid  at the origin. 
\vskip 2mm	
	(i) \begin{equation*}
		\lim\limits_{p \rightarrow \infty} \left(\frac{\omega_{m, k, \vec{i}}^p (K)}{\omega^\infty_{m, k, \vec{i}}(K)}\right)^{n+p} = \textup{exp} \left(-\frac{n \, D_{KL} (P_K||Q_K)}{\mu_{\vec{i}} - \vol_n(K^\circ)}\right).
	\end{equation*}
\vskip 2mm
(ii) \begin{equation*}
	\lim\limits_{p \rightarrow 0} \left(\frac{\omega_{m, k, \vec{i}}^p (K^\circ)}{\omega^0_{m, k, \vec{i}}(K^\circ)}\right)^{\frac{n(n+p)}p} = \textup{exp} \left(-\frac{n \, D_{KL} (P_{K^\circ}||Q_{K^\circ})}{\mu_{\vec{i}} - \vol_n(K^\circ)}\right).
\end{equation*}
\end{theorem}

\begin{Proof}
(i)	We use l'Hospitals rule
	
	\begin{align*}
		&\lim\limits_{p \rightarrow \infty} \ln \left(\left(\frac{\omega_{m, k, \vec{i}}^p (K)}{\omega^\infty_{m, k, \vec{i}}(K)}\right)^{n+p} \right) = \lim\limits_{p \rightarrow \infty} \frac{\ln \left(\frac{\omega_{m, k, \vec{i}}^p (K)}{\omega^\infty_{m, k, \vec{i}}(K)} \right)}{(n+p)^{-1}}  = - \lim\limits_{p \rightarrow \infty} (n+p)^2 \frac{\frac{d}{dp} \left(\omega_{m, k, \vec{i}}^p (K)\right) }{\omega_{m, k, \vec{i}}^p (K)} \\
		&= -\lim\limits_{p \rightarrow \infty} \frac{(n+p)^2}{\omega_{m, k, \vec{i}}^p (K)}  \int\limits_{\partial K} \frac{d}{dp} \left( \textup{exp} \left( \ln \left( \frac{H_{n-1}(x)^{\frac{p}{n+p}}}
		{\langle x,N_K(x)\rangle ^{\frac{n(p-1)}{n+p}}} \right) \right) \right) d\mu_{\vec{i}} (x)\\
		&=   -\lim\limits_{p \rightarrow \infty} \frac{n}{\omega_{m, k, \vec{i}}^p (K)}  \int\limits_{\partial K}  \frac{H_{n-1}(x)^{\frac{p}{n+p}}}
		{\langle x,N(x)\rangle ^{\frac{n(p-1)}{n+p}}}  \ln{\left( \frac{H_{n-1}(x)}{\langle x,N_K(x)\rangle^{n+1}}\right)} d\mu_{\vec{i}} (x)\\
		&=  - \frac{n}{\omega_{m, k, \vec{i}}^\infty (K)} \int\limits_{\partial K} \frac{H_{n-1}(x)}
		{\langle x,N_K(x)\rangle ^n} \ln{\left( \frac{H_{n-1}(x)}{\langle x,N_K(x)\rangle^{n+1}}\right)} d\mu_{\vec{i}} (x)\\
		&= - \frac{n}{\omega_{m, k, \vec{i}}^\infty (K)} D_{KL}(P_K||Q_K).
	\end{align*}

(ii)	We use l'Hospitals rule
\begin{align*}
	&\lim\limits_{p \rightarrow 0} \ln \left( \left(\frac{\omega_{m, k, \vec{i}}^p (K^\circ)}{\omega^0_{m, k, \vec{i}}(K^\circ)}\right)^{\frac{n(n+p)}p} \right) = \lim\limits_{p \rightarrow 0}  n \frac{\ln \left(\frac{\omega_{m, k, \vec{i}}^p (K^\circ)}{\omega^0_{m, k, \vec{i}}(K^\circ)} \right)}{p(n+p)^{-1}} = \lim\limits_{p \rightarrow 0}  (n+p)^2 \frac{\frac{d}{dp}\omega_{m, k, \vec{i}}^p (K^\circ)}{\omega^p_{m, k, \vec{i}}(K^\circ)}\\
	&= \lim\limits_{p \rightarrow 0}  \frac{(n+p)^2}{\omega^p_{m, k, \vec{i}}(K^\circ)}   \int\limits_{\partial K^\circ} \frac{d}{dp} \left( \textup{exp} \left( \ln \left( \frac{H_{n-1}(x)^{\frac{p}{n+p}}}
	{\langle x,N_{K^\circ}(x)\rangle ^{\frac{n(p-1)}{n+p}}} \right) \right) \right) d\mu_{\vec{i}} (x)\\
	&= \lim\limits_{p \rightarrow 0}  \frac{n}{\omega^p_{m, k, \vec{i}}(K^\circ)} \int\limits_{\partial K^\circ}  \frac{H_{n-1}(x)^{\frac{p}{n+p}}}
	{\langle x,N_{K^\circ}(x)\rangle ^{\frac{n(p-1)}{n+p}}}  \ln{\left( \frac{H_{n-1}(x)}{\langle x,N_{K^\circ}(x)\rangle^{n+1}}\right)} d\mu_{\vec{i}} (x)\\
	&=  \frac{n}{\omega^0_{m, k, \vec{i}}(K^\circ)} \int\limits_{\partial K^\circ}  \langle x,N_{K^\circ}(x)\rangle  \ln{\left( \frac{H_{n-1}(x)}{\langle x,N_{K^\circ}(x)\rangle^{n+1}}\right)} d\mu_{\vec{i}} (x)\\
	&= - \frac{n}{\omega^0_{m, k, \vec{i}}(K^\circ)} D_{KL}(Q_{K^\circ}||P_{K^\circ}).
\end{align*}

\end{Proof}

\vskip 3mm
\section{ $f$-divergence} \label{Section-fdiv}

The results in the previous section lead naturally to consider more general $f$-divergences than just the Kullback-Leibler divergence.

\subsection{Background on $f$-divergence.} \label{subsection:f-div}
In information theory, probability theory and statistics, an
$f$-divergence is a function that measures the difference between two (probability)
distributions. This notion was introduced by 
Csisz\'ar \cite{Csiszar}, and independently Morimoto \cite{Morimoto1963} and Ali \& Silvery \cite{AliSilvery1966}.
\par
Let $(X, \mu)$ be a measure space  and let  $P=p \mu$ and  $Q=q \mu$ be  (probability) measures on $X$ that are  absolutely continuous with respect to the measure $\mu$.  
Let $f: (0, \infty) \rightarrow  \mathbb{R}$ be a convex  or a concave  function.
The $*$-adjoint function $f^*:(0, \infty) \rightarrow  \mathbb{R}$ of $f$  is defined by 
\begin{equation}\label{adjoint}
f^*(t) = t f (1/t), \ \  t\in(0, \infty).
\end{equation}
\par
It is obvious   that $(f^*)^*=f$ and that $f^*$ is again convex  if $f$ is convex,  respectively concave if $f$ is concave.
Then the $f$-divergence   $D_f(P,Q)$ of the measures $P$ and $Q$ is defined by 
\begin{eqnarray}\label{def:fdiv1}
D_f(P,Q)&=&
 \int_{\{pq>0\} }f\left(\frac{p}{q} \right) q d\mu + f(0)\  Q\left(\{x\in X: p(x) =0\}\right)\nonumber \\
 &+& f^*(0) \ P\left(\{x\in X: q(x) =0\}\right),
\end{eqnarray}
provided the expressions exist. Here 
\begin{equation}\label{fat0}
f(0) = \lim_{t\downarrow 0} f(t)  \  \  \text{ and} \  \   f^*(0) = \lim_{t\downarrow 0} f^*(t).
\end{equation}
We make the convention that $0 \cdot \infty =0$. 
\vskip 2mm
Please note that 
\begin{equation}\label{fstern}
D_f(P,Q)=D_{f^*}(Q,P).
\end{equation}
With (\ref{fat0}) and as  
$$f^*(0) \ P\left(\{x\in X: q(x) =0\}\right) = \int _{\{q=0\}} f^*\left(\frac{q}{p} \right) p  d\mu =  \int _{\{q=0\}} f\left(\frac{p}{q} \right) q  d\mu,$$
we can write in short
\begin{equation}\label{def:fdiv2}
D_f(P,Q)=
 \int_{X} f\left(\frac{p}{q} \right) q d\mu.
\end{equation}
\vskip 3mm
\noindent
Examples of $f$-divergences are as follows.
\par
\noindent
1.  For $f(t) = t \ln t$ (with  $*$-adjoint function $f^*(t) = - \ln t$), the $f$-divergence is   {\em Kullback-Leibler divergence} or {\em relative entropy} from $P$ to $Q$ (see \cite{CoverThomas})
\begin{equation}\label{relent}
 D_{KL}(P\|Q)= \int_{X} p \ln \frac{p}{q} d\mu.
\end{equation}
\par
\noindent
2. For   the convex or concave functions  $f(t) = t^\alpha$ we obtain the {\em Hellinger integrals} (e.g. \cite{LieseVajda2006})
\begin{equation}\label{Hellinger}
H_\alpha (P,Q) = \int _X  p^\alpha q^{1-\alpha} d\mu.
\end{equation}
Those are related to the 
R\'enyi divergence of order $\alpha$, $\alpha \neq 1$,  introduced by  R\'enyi \cite{Ren} (for $\alpha >0$) as 
\begin{equation}\label{renyi}
D_\alpha(P\|Q)=
\frac{1}{\alpha -1} \ln \left( \int_X p^\alpha q^{1-\alpha} d\mu \right)= \frac{1}{\alpha -1} \ln \left( H_\alpha (P,Q)\right).
\end{equation}
The case $\alpha =1$ is the relative entropy $ D_{KL}(P\|Q)$.
\vskip 2mm
\noindent
More on $f$-divergence can be found in e.g. \cite{CaglarWerner2014, CaglarWerner2015, CaglarKolesnikovWerner, LieseVajda2006,  Werner2012/1, Werner2012, WernerYe2013}.

\subsection{$f$-divergence for the $\mu_{\vec{i}}$-measure}

In \cite{Werner2012}, $f$-divergence with respect to the surface area measure $\mu_K$  was introduced for a  convex body $K$ in $\mathbb{R}^n$ with $0$ in its interior.
We now introduce similarly $f$-divergence with respect to the  measure $\mu_{\vec{i}}$.
\vskip 2mm
\begin{defi}\label{defn:f-div-bodies}
Let $f: (0, \infty) \rightarrow \mathbb{R}$ be  a  convex or  concave function. Let $p_K$ and $q_K$ be as in (\ref{densities}).
Then the $f$-divergence $D_f \left(P_K, Q_K \right)$ of a convex body $K$ in $\mathbb{R}^n$ with respect to the $\mu_{\vec{i}}$-measure is
\begin{eqnarray*}
D_f \left(P_K, Q_K \right)&=&
 \int_{\partial K } f  \left(\frac{p_K}{q_K}\right) q_K d \mu_{\vec{i}}\\
 &=&\int_{\partial K } f  \left( \frac{H_{n-1}}{\langle x,N(x)\rangle^{n+1}} \right) \langle x,N(x)\rangle  d \mu_{\vec{i}}
\end{eqnarray*}
\end{defi}
\par
\noindent
In particular, for $f(t)=t \log t$ we recover the above Kullback-Leibler divergences of Section \ref{KL}  and for  $f(t) = t^\frac{p}{n+p}$, we obtain 
the weighted $L_p$-affine surface areas.
\vskip 2mm
\noindent
{\bf Remarks}
\vskip 2mm
\noindent
(i) By  (\ref{fstern}),   
\begin{eqnarray}\label{f-div2,0}
 D_f(Q_K,P_K)&=& \int _{\partial K} f\left(\frac{q_K}{p_K}\right) p_K d\mu_{\vec{i}}
= D_{f^*}(P_K, Q_K)  \nonumber \\
&= &\int _{\partial K} f^*\left(\frac{p_K}{q_K}\right) q_K d\mu_{\vec{i}}  \nonumber \\
&=& \int _{\partial K} f\left(\frac{ \vol_n( K^\circ )\langle x, N_K(x) \rangle ^{n+1}}{\vol_n(K) H_{n-1} (x)} \right)\frac{H_{n-1}(x)\  d\mathcal{H}^{n-1}} {n\,  \vol_n(K^\circ) \langle x, N_K(x) \rangle ^{n}}. 
\end{eqnarray}
\vskip 2mm
\noindent
(ii) $f$-divergences can also be expressed as integrals over $S^{n-1}$, 
\begin{eqnarray}\label{f-div1,2}
D_f(P_K, Q_K)&= &
\int _{S^{n-1}} f\left(\frac{\vol_n(K) } {\vol_n(K^\circ) f_K(u) h_K(u) ^{n+1}}\right)\frac{ h_K(u)f_K(u)}{n\,  \vol_n(K)} d\sigma 
\end{eqnarray}
and
\vskip 2mm
\noindent
\begin{eqnarray}\label{f-div2,2}
D_f(Q_K,P_K)&=& 
\int _{ S^{n-1}} f\left(\frac{ \vol_n(K^\circ) f_K(u) h_K(u) ^{n+1}}{\vol_n(K) } \right)\frac{ d\sigma_K} {n\,  \vol_n(K^\circ )h_K(u)^{n}}.
\end{eqnarray}
\vskip 2mm
\noindent
(iii) 
Similar to  (\ref{f-div1,2}) and (\ref{f-div2,2}), one can define mixed $f$-divergences for $n$ convex bodies in $\mathbb{R}^n$. We will not treat those here but  will concentrate on $f$-divergence for one convex body.  
We also refer to \cite{Ye2012}, where they have been investigated for functions in $Conv(0,\infty)$.

\begin{prop} \label{propos}
Let $f: (0, \infty) \rightarrow \mathbb{R}$ be a concave function.
Then
\begin{eqnarray*}
D_f \left(P_K, Q_K \right)\leq f\left(\frac{\mu_{\vec{i}}-\vol_n(K^{\circ}) }{\mu_{\vec{i}}-\vol_n(K) } \right) \, \mu_{\vec{i}}-\vol_n(K).
 \end{eqnarray*}
If  $f$ is convex, the inequality is reversed. 
If $f$ is linear, equality holds in (\ref{propos}).
If $f$ is not linear, equality  holds 
iff $K$ is an ellipsoid.

\end{prop}

\begin{proof}
Let $f: (0, \infty) \rightarrow \mathbb{R}$ be a concave function. By Jensen's inequality 
\begin{eqnarray*}
D_f \left(P_K, Q_K \right) &=& \left(\int_{\partial K } f  \left( \frac{H_{n-1}}{\langle x,N(x)\rangle^{n+1}} \right) \langle x,N(x)\rangle  \, \frac{d \mu_{\vec{i}} }{\mu_{\vec{i}}-\vol_n(K)} \right) \, \mu_{\vec{i}}-\vol_n(K)\\
&\leq& f\left(\int_{\partial K } \frac{H_{n-1}}{\langle x,N(x)\rangle^{n+1}} \langle x,N(x) \rangle \frac{  d \mu_{\vec{i}}} { \mu_{\vec{i}}-\vol_n(K)} \right)   \, \mu_{\vec{i}}-\vol_n(K)
\\&=&
f\left(\frac{\mu_{\vec{i}}-\vol_n(K^{\circ} }{\mu_{\vec{i}}-\vol_n(K) } \right) \, \mu_{\vec{i}}-\vol_n(K).
 \end{eqnarray*}
Equality holds in Jensen's inequality iff either $f$ is linear or $\frac{p_K}{q_K}$ is constant.
Indeed, if $f(t) = at +b$, then 
\begin{eqnarray*}
D_f(P_K, Q_K)&=& 
\int _{\partial K  } \left(a \frac{p_K}{q_K} +b \right) q_K \, d \mu_{\vec{i}} 
= a \int _{\partial K  } p_K \,  d \mu_{\vec{i}} + b  \int _{\partial K  } q_K \,  d \mu_{\vec{i}} \\
&= &a \, \mu_{\vec{i}}-\vol_n(K^{\circ})  +b \, \mu_{\vec{i}}-\vol_n(K).
\end{eqnarray*}
If $f$ is not linear, equality holds iff $\frac{p_K}{q_K}=c$,  $c$ is a  constant. 
By the above  quoted Theorem \ref{Petty} of Petty, this holds iff $K$ is an ellipsoid.
Note that the constant $c$ is different from $0$, as we assume that  $K$ is  $C^2_+$.
\end{proof} 
\vskip 3mm
\noindent
Now we show that $D_f({P}_K,{Q}_K)$  are valuations, i.e.
for convex bodies $K$ and $L$  such that 
that $K \cup L $ is convex
\begin{equation}\label{val}
D_f({P}_{K\cup L},{Q}_{K\cup L}) + D_f({P}_{K\cap L},{Q}_{K\cap L}) = D_f({P}_K,{Q}_K) + D_f({P}_L,{Q}_L).
\end{equation}
\vskip 3mm
\noindent
\begin{prop}\label{prop3}
Let $K$ be a convex body in $\R^n$ with the origin in its interior and let  $f: (0, \infty) \rightarrow \mathbb{R}$ be a convex function. Then $D_f(P_K,Q_K)$ are valuations.
\end{prop}
\vskip 2mm
\noindent
{\bf Proof.} 
To prove (\ref{val}), we proceed as in
Sch\"utt  \cite{Schuett1993}. For completeness, we  include the argument. We decompose
$$
\partial (K \cup L) = (\partial K \cap \partial L)  \cup (\partial K \cap  L^c) \cup ( K^c \cap \partial L),
$$
$$
\partial (K \cap L) = (\partial K \cap \partial L)  \cup (\partial K \cap \text{int} L) \cup (\text{int} K \cap \partial L),
$$
$$
\partial K = (\partial K \cap \partial L)  \cup (\partial K \cap L^c) \cup (\partial K \cap \text{int} L),
$$
$$
\partial L = (\partial K \cap \partial L)  \cup (\partial K^c \cap \partial L) \cup (\text{int} K \cap \partial L),
$$
where all unions on the right hand side are disjoint. Note that for $x$ such
that the curvatures $\kappa_K(x)$, 
$\kappa_L(x)$, $\kappa_{K\cup L} (x)$ and $\kappa_{K\cap L} (x)$ exist,
\begin{equation}\label{support}
\langle x, N_K(x)\rangle  = \langle x, N_L(x) \rangle =  \langle x, N_{K\cap L}(x) \rangle = \langle x,  N_{K\cup L}(x) \rangle
\end{equation}
and
\begin{equation}\label{curv}
\kappa_{K\cup L}(x)= \min\{\kappa_K(x), \kappa_L(x)\}, \  \  \kappa_{K\cap L} (x)= \max\{\kappa_K(x), \kappa_L(x)\}.
\end{equation}
To prove (\ref{val}), we split the involved integral using the above decompositions  
and  (\ref{support}) and (\ref{curv}).

\vskip 5mm

\begin{center}
		\begin{tabular}{l l}
			\multirow{5}{17em}{\addressT} & \qquad \qquad  \multirow{5}{17em}{\addressW}
		\end{tabular}
	\end{center}

\end{document}